\documentclass[1 [leqno,11pt]{article} 
\author{Stefano Scrobogna}
\title{Some remark on the existence of infinitely many nonphysical solutions to the incompressible Navier-Stokes equations}

\usepackage[a4paper]{geometry}
\usepackage{empheq}
\usepackage{times}
\usepackage[T1]{fontenc}
\usepackage{amssymb, amsmath,  bm}
\usepackage{amsfonts}
\usepackage[english]{babel}
\usepackage{amssymb,amsthm}
\usepackage{mathtools}
\DeclareMathAlphabet{\mathcal}{OMS}{cmsy}{m}{n}
\usepackage{hyperref}
\usepackage{cite}
\allowdisplaybreaks[1]
\usepackage[bbgreekl]{mathbbol}
\DeclareSymbolFontAlphabet{\mathbb}{AMSb}
\DeclareSymbolFontAlphabet{\mathbbl}{bbold}
\usepackage{xfrac}

\renewcommand{\d}{\textnormal{d}}
\renewcommand{\div}{\textnormal{div}}
\newcommand{\fine}{\hfill$\blacklozenge$}

\newcommand{\pare}[1]{\left( #1 \right)}

\newcommand{\av}[1]{\left| #1 \right|}
\newcommand{\bra}[1]{\left[ #1 \right]}
\newcommand{\set}[1]{\left\{ #1 \right\}}

\newcommand{\cC}{\mathcal{C}}

\newcommand{\cJ}{\mathcal{J}}
\newcommand{\cD}{\mathcal{D}}

\newcommand{\bR}{\mathbb{R}}
\newcommand{\bT}{\mathbb{T}}
\newcommand{\bZ}{\mathbb{Z}}

\newcommand{\bN}{\mathbb{N}}

\newcommand{\loc}{\textnormal{loc}}
\newcommand{\NS}{Navier-Stokes }

\theoremstyle{theorem}
\newtheorem{theorem}{Theorem}[section]
\newtheorem*{theorem*}{Theorem}
\newtheorem{prop}[theorem]{Proposition}

\theoremstyle{definition}

\newtheorem{rem}[theorem]{Remark}

\newtheorem{claim}{Claim}

\numberwithin{equation}{section}

\begin{document}

\AtEndDocument{\bigskip{\footnotesize
  \textsc{BCAM - Basque Center for Applied Mathematics,Mazarredo, 14,  E48009 Bilbao, Basque Country -- Spain} \par
  \textit{E-mail address:}  \texttt{\href{mailto:sscrobogna@bcamath.org}{sscrobogna@bcamath.org}}}}

 \maketitle
 
 \begin{abstract}
 We prove that there exist infinitely many distributional solutions with infinite kinetic energy to  the incompressible  \NS\ equations in $ \bR^2 $. We prove as well the existence of infinitely many distributional solutions for Burgers equation in $ \bR $. 
 \end{abstract}
 
 In the present work we consider the incompressible \NS\ equations 
 \begin{equation}\tag{INS}\label{NS}
 \left\lbrace
 \begin{aligned}
 & \partial_t u + u\cdot \nabla u -\Delta u = -\nabla p,  \\
 & \div\ u = 0, \\
 & \left. u\right|_{t=0}=u_0, 
 \end{aligned}
 \right.
 \end{equation}
 in the whole bidimensional space $ \bR^2 $. Very recently in \cite{BV17} T. Buckmaster and V. Vicol proved the existence of infinitely many periodic weak solutions of \eqref{NS} in $ \bT^3 = \sfrac{\bR^3}{\bZ^3} $ with finite kinetic energy using the technique of convex integration developed by C. De Lellis and L. Sz\'ekelyhidi in \cite{dLS09} and \cite{dLS13} (and later used in order to obtain several other outstanding results \cite{dLS10}, \cite{dLS14}, \cite{bdLIS15}, \cite{I16}, \cite{BdLSV17}). Unfortunately the result in \cite{BV17}  does not solve the longstanding conjectures posed by J. Leray in 1934 on whether weak solutions which belong to the energy space 
 \begin{equation*}
 L^\infty\pare{\bR_+; L^2\pare{\bR^d}}\cap L^2 \pare{\bR_+; \dot{H}^1\pare{\bR^d}}, 
 \end{equation*}
 are unique in such space when $ d\geqslant 3 $. \\
 
 In such direction, when $ d=3 $, the Escauriaza-Seregin-\v Sver\'ak criterion \cite{ISS03} provides a sharp characterization of smoothness of Leray-Hopf solutions to the \NS\ equation:  if 	the  Leray-Hopf solutions belong as well to the space $ L^\infty\pare{\bra{0, t}; L^{3}\pare{\bR^3}} $ for some $ t>0 $ then the solution is unique and smooth up to time $ t $. \\
 
 The result of Buckmaster and Vicol must hence be understood in contraposition to the result proved in \cite{ISS03} (which is the endpoint result of the more general Lady\v{z}enskaya-Prodi-Serrin regularity criterion \cite{KL57}, \cite{P59}, \cite{S62}): they prove in fact that there exist a $ \beta \in \pare{0, \sfrac{1}{3}} $ such that for any nonnegative smooth function $ e: \bra{0, T}\to \bR_+ $ there exists a $ v \in \cC \pare{\bra{0, T}; H^{\beta} \pare{\bT^3}} $ weak solution of \eqref{NS} such that $ e $ is the kinetic energy profile of $ v $, i.e. $ e\pare{t}= \int _{\bT^3} \av{v\pare{x, t}}^2\d x $. Indeed there is hope to achieve nonuniqueness for weak solutions only if  $ \beta <\sfrac{1}{2} $, since otherwise, by Sobolev embeddings, $ v $ would be in $L^\infty\pare{\bra{0, t}; L^3\pare{\bT^3}} $ and hence thanks to the result proved in \cite{ISS03} it would be smooth, and being so his energy decay would be unique and not arbitrary as it is proved in \cite{BV17}. \\
 
 n the present note we prove a much weaker result thank the one proved in  \cite{BV17}; we prove that in $ \bR^d, \ d \geqslant 2 $ there exists infinitely many nontrivial smooth solutions of the initial value problem \eqref{NS} when $ u_0=0 $ with \textit{infinite kinetic energy} for each $ t>0 $.  In particular the functions we construct, following the methods of the classical result of Tychonoff \cite{T35}, are \textit{not tempered distributions} and they grow exponentially, as $ \av{x}\to \infty $. Hence we provide an elementary example of nonuniqueness of distributional solutions for the system \eqref{NS} when the space dimension is two, result which is not proved in \cite{BV17}. In Section \ref{sec:Burgers} we prove that Burgers equation in $ \bR $ admits infinitely many distributional solutions exploiting the same technique. \\

 \section{Tychonoff's example}\label{sec:Tyc}
 
 In this section we illustrate the well known methodology exploited by Tychonoff in 1935 in order to prove that the one-dimensional heat equation \eqref{eq:heat1D} with zero initial data admits infinitely many weak solutions in the domain $ \pare{x, t}\in\bR\times \bR_+ $. In the original work \cite{T35} several domains and boundary conditions are considered, nonetheless we will provide a proof of such result in the restricted setting mentioned above since our purpose is to provide a simple introduction to the method which we will exploit in the following. \\
 
 A. Tychonoff in \cite{T35} proved that the homogeneous, one dimensional, linear diffusion equation
 \begin{equation}\label{eq:heat1D}
 \left\lbrace
 \begin{aligned}
 & u_t = u_{xx}, & \pare{x, t} & \in \bR\times\pare{0, \infty},\\
 & u\pare{x, 0} =0, & x &\in \bR, 
 \end{aligned}
 \right. 
 \end{equation}
 admits infinitely many smooth  solutions which does not decay as $ \av{x} $ tends to infinity. The method he uses is very simple: he looks for solutions $ u=u\pare{x, t} $ of the form
 \begin{equation*}
 u\pare{x, t}=\sum_{n=0}^\infty a_n f^{\pare{n}}\pare{t} \ x^n, 
 \end{equation*}
 where $ f $ is defined on $ \bR_+ $, infinitely  differentiable in such domain and such that\footnote{Here and in  the rest of the work $ f^{\pare{n}} $ denotes the $ n $-th derivative of $ f $. } $ f^{\pare{n}}\pare{t}\xrightarrow{t\to 0^+} 0 $ for any $ n $. Since there exist many $ f $ satisfying such property we will deduce the nonuniqueness result letting varying $ f $ in an infinite family. \\

 Let us hence for instance define the following sequence $ \pare{u_k}_{k\geqslant 1} $ of formal series
 \begin{equation}\label{eq:heat_1d_solution}
 u_k\pare{x, t} = \sum_{n=0}^\infty \frac{\pare{\exp\set{-\displaystyle \frac{1}{t^{2k}}}}^{\pare{n}}}{\pare{2n}!}\ x^{2 n}, \hspace{5mm}\text{ for any } k\in\bN, \ k\geqslant 1. 
 \end{equation}
 
 \noindent It is possible to prove that, fixed a $ t>0 $,  there exists a $ \theta = \theta \pare{2k} > 0 $ such that
 \begin{equation}	\label{eq:bound_fkn}
 \av{\pare{\exp\set{-\displaystyle \frac{1}{t^{2k}}}}^{\pare{n}}} \leqslant \frac{n!}{\pare{\theta t}^n} \ e^{-\frac{1}{2} \ t^{-2k}}, 
 \end{equation}

 \noindent we can hence deduce that, fixed $ t>0 $,  the series \eqref{eq:heat_1d_solution} is convergent in any compact set of $ \bR $ and is hence a genuine pointwise solution of  \eqref{eq:heat1D} in any compact set, which implies as well that it is a distributional solution of \eqref{eq:heat1D}. Moreover 
\begin{equation*}
u_k\pare{\cdot , t} \xrightarrow{t \searrow 0} 0 , \text{ in } \cD'\pare{\bR^2}. 
\end{equation*} 
  Indeed its growth as $ \av{x}\to \infty $ is greater than any power law, hence it is not a tempered distribution.

\section{Infinitely many  distributional solutions of \eqref{NS} in $ \bR^2\times \bR_+ $}
In this section we exploit the construction illustrated in the previous section in order to construct infinitely many distributional  solutions for the Cauchy problem

\begin{equation}\label{NS0}
 \left\lbrace
 \begin{aligned}
 & \partial_t u + u\cdot \nabla u -\Delta u = -\nabla p,  \\
 & \div\ u = 0, \\
 & \left. u\right|_{t=0}=0, 
 \end{aligned}
 \right.
 \end{equation}
 when $ \pare{x, t}\in {\bR^2\times\bR_+} $. \\
 
 The main result of the present note is the following one:
 
 \begin{theorem}\label{thm:existence_infinitely_many_sol_R2}
 There exist infinitely many smooth
 \begin{equation*}
 u \in \cC^{1}\pare{\pare{0, \infty} ; \cD'\pare{\bR^2}} \cap L^\infty_{\loc}\pare{{\bR^2}\times \bR_+}, 
 \end{equation*}
distributional solutions of the two dimensional incompressible \NS equations \eqref{NS0} such that
\begin{equation*}
u\pare{\cdot , t} \xrightarrow{t \ \searrow \ 0} 0 \hspace{5mm}\text{ in } \cD'\pare{\bR^2}. 
\end{equation*}
 \end{theorem}

 \begin{rem}
 \begin{itemize}

 \item[$ \diamond $] It is important to underline the fact that each $ u $ pointwise solution of \eqref{NS0} will \textit{not be} a tempered distribution. In particular the following relation will hold true
 \begin{equation*}
 \frac{\av{ \partial^\alpha u\pare{x, t}}}{\av{x}^N}\xrightarrow{\av{x}\to\infty} \infty, 
 \end{equation*}
 for any multi-index $ \alpha $ and $ N>0 $. 
 
  \item[$ \diamond $] It is obvious from the above point that the solutions constructed  are not $ L^2 $ integrable, hence they are not Leray-Hopf solutions.  It may be of interest though to remark that they are neither uniformly locally $ L^2 $ integrable, whence they are not infinite energy weak solutions in the sense of Lemari\'e-Rieusset \cite{LR99}. 
  
  \item[$ \diamond $] We want to underline that every solution of \eqref{NS0} will be \textit{of infinite kinetic energy for each} $ t>0 $. 
   \fine

 \end{itemize}
 \end{rem}
 
 Before starting to prove Theorem \ref{thm:existence_infinitely_many_sol_R2} let us recall some classical result concerning the incompressible \NS\ equations. Given a bi-dimensional  vector field $ u $ we can define the \textit{vorticity} as 
 \begin{equation}\label{eq:vorticity}
 \omega\pare{x, t} = -\partial_2 u_1 \pare{x, t} + \partial_1 u_2\pare{x, t}. 
 \end{equation}
 If $ u $ is a smooth solution of the equation \eqref{NS0} then $ \omega $ has to solve 
 \begin{equation} \label{eq:equation_vorticity_2D}
 \left\lbrace
 \begin{aligned}
 & \partial_t \omega + u \cdot \nabla \omega -\Delta \omega =0, \\
 & \omega\pare{x, 0}=0, 
 \end{aligned}
 \right.
 \end{equation}
 and we can recover $ u $ from $ \omega $ via the Biot-Savart law
 \begin{equation*}
 u\pare{x, t} = \frac{1}{2\pi} \int _{\bR^2} \frac{\pare{x-y}^\perp}{\av{x-y}} \ \omega\pare{y, t} \d y,
 \end{equation*}
 see \cite[p. 292]{BCD11}. \\
 Moreover since $ \div\ u =0 $ there exists a unique (up to an additive constant) \textit{stream function} $ \psi\pare{x, t} $ such that
 \begin{equation}\label{eq:stream}
 u \pare{x, t} = \pare{\begin{array}{c}
 -\partial_2 \psi \pare{x, t} \\ \partial_1 \psi \pare{x, t}
 \end{array}} = \nabla ^\perp \psi \pare{x, t}. 
 \end{equation}
 Comparing hence the equations \eqref{eq:vorticity} and \eqref{eq:stream} we deduce the Poisson equation for $ \psi $
 \begin{equation}\label{eq:Poisson_stream}
 \Delta \psi = \omega. 
 \end{equation}

 Let us consider hence at this point a radial symmetric smooth vorticity $ \omega $, i.e.
 \begin{equation*}
 \omega = \omega\pare{r}, \hspace{1cm} r = \av{x}=\sqrt{x_1^2 + x_2^2}.
 \end{equation*}
 Since the Laplace operator is rotationally invariant we can deduce, as in \cite[pp. 47, 48]{MB02}, that $ \psi $ is a radially symmetric function as well, whence \eqref{eq:stream} becomes
 \begin{equation}\label{eq:stream_radial}
 u = \frac{1}{r}\ x^\perp \psi_r, \hspace{1cm} x^\perp=\pare{-x_2, x_1}. 
 \end{equation}
 
 The Poisson equation \eqref{eq:Poisson_stream} in polar coordinates reads as 
 \begin{equation*}
 \psi_{rr}+ \frac{1}{r}\ \psi _r = \omega, 
 \end{equation*}
 whence
 \begin{equation}\label{eq:first_der_stream}
 \psi_r\pare{r, t} = \frac{1}{r}\int_0^r s \ \omega\pare{s, t} \d s, 
 \end{equation}
 whence comparing \eqref{eq:stream_radial} and \eqref{eq:first_der_stream} we deduce the following simplified Biot-Savart law for radial vorticities; 
 \begin{equation}\label{eq:BS_radial}
 u\pare{x, t} = \frac{1}{r^2}\ x^\perp\int_0^r s \ \omega\pare{s, t} \d s. 
 \end{equation}

 Let us observe now that if \eqref{eq:stream} and \eqref{eq:Poisson_stream} hold true  the following  identity holds true:
 \begin{equation*}
 u\cdot\nabla \omega = \det \pare{\begin{array}{cc}
 \partial_1 \psi & \partial_2 \psi \\
 \partial_1 \Delta \psi & \partial_2 \Delta \psi
 \end{array}} = \cJ\pare{\psi, \Delta \psi}, 
 \end{equation*}
 but if the stream function $ \psi $ is radial then
 \begin{equation*}
 \cJ\pare{\psi, \Delta \psi}\equiv 0. 
 \end{equation*}
Let us sketch a quick proof of such identity; let us recall that for the radial function $ \psi $
\begin{equation*}
\left\lbrace
\begin{aligned}
& \partial_1 \psi = \cos \varphi \ \psi_r, \\
& \partial_2 \psi = \sin \varphi \ \psi_r, 
\end{aligned}
\right.
\end{equation*}
and since $ \psi $ is radial
\begin{equation*}
\Delta \psi = \psi_{rr} + \frac{1}{r} \psi_r = \Psi, 
\end{equation*}
which is again radial, whence
\begin{equation*}
	\cJ \pare{\psi, \Delta \psi} = \det\pare{
		\begin{array}{cc}
		\cos \varphi \ \psi_r & \sin \varphi \ \psi_r \\
		\cos \varphi \ \Psi_r & \sin \varphi \ \Psi_r 
		\end{array}			
	} =0.
\end{equation*}
 Whence the vorticity $ \omega $ solves the linear homogeneous diffusive equation
\begin{equation}\label{eq:heat_vorticity}
\partial_t\omega = \Delta \omega,
\end{equation} 
  in $ \bR^2 $. \\
  
  Considering hence the radial Biot-Savart law \eqref{eq:BS_radial} in order to prove Theorem \ref{thm:existence_infinitely_many_sol_R2} is will be sufficient to prove the following result
  
  \begin{prop}\label{prop:construction_vorticity}
  There exist infinitely many $ \omega \in \cC^1 \pare{\pare{0, \infty} ; \cD'\pare{\bR^2}} \cap L^\infty_{\loc}\pare{{\bR^2}\times \bR_+} $ smooth radial distributional solutions to linear homogeneous diffusive equation
  \begin{equation}\label{eq:vorticity_polar_CP}
  \begin{aligned}
  & \omega_t -\frac{1}{r}\ \omega_r - \omega_{rr}=0, & \pare{r, t} & \in \bR_+\times \pare{0, \infty},  
  \end{aligned}
  \end{equation}
  such that $ \omega\pare{\cdot , t}\to 0 $ in $ \cD'\pare{\bR_+} $ as $ t\to 0 $. 
  \end{prop}

\begin{rem}
The equation \eqref{eq:vorticity_polar_CP} is nothing but \eqref{eq:heat_vorticity} in polar coordinates for radial functions. \fine
\end{rem}

  \begin{proof}
 Following Tychonoff's method let us construct a smooth solution of \eqref{eq:vorticity_polar_CP} in the form
  \begin{equation}\label{eq:omega_power_law}
  \omega\pare{r, t} = \sum_{n=0}^\infty a_n f^{\left( n \right)}\pare{t} \ r^{2n}, 
  \end{equation}
  and imposing that
  \begin{equation*}
  \lim _{t\to 0^+} f^{\pare{n}}\pare{t}=0, \hspace{1cm}\forall \ n\in \bN. 
  \end{equation*}
  Standard computations shows that
  \begin{align*}
  \omega_t\pare{r, t} & = \sum_{n=0}^\infty a_n f^{\left( n+1 \right)}\pare{t} \ r^{2n}, \\
  %-------------------------------------
  \frac{1}{r}\ \omega_r \pare{r, t} & = \sum_{n=0}^\infty \pare{2n + 1} a_{n+1} f^{\left( n +1 \right)}\pare{t} \  r^{2n}, \\
  %--------------------------------
  \omega_{rr}\pare{r, t} & = \sum_{n=0}^\infty 2\pare{n+1}\pare{2n + 1} a_{n+1} f^{\left( n +1 \right)}\pare{t} \  r^{2n}. 
  \end{align*}
  Whence it is sufficient to define the sequence $ \pare{a_n}_n $ recursively as
  \begin{equation}\label{eq:sequence_an}
  a_{n+1} = \frac{a_n}{\pare{2n +1}\pare{2n+3}} = \pare{\prod_{k=1}^{n+1} \pare{2k-1}^{-2}}\frac{a_0}{2\pare{n+1}+1}.
  \end{equation}

   Let us note that  until now no initial condition is defined, in order to do so for each $ k\in \bN, \ k\geqslant 1 $ let us define
 \begin{equation}\label{eq:def_fk}
 f_k\pare{t} = \exp \set{- \frac{1}{t^{2k}}}, 
 \end{equation}
 and let the sequence $ \pare{a_n}_{n\geqslant 0} $ satisfy \eqref{eq:sequence_an} for any $ a_0 > 0 $. 
 Let us now consider the sequence (indexed in $ k\geqslant 1 $)
 \begin{equation*}
 \omega_k\pare{r, t} = \sum_{n=0}^\infty a_n f_k ^{\pare{n}}\pare{t} \ r^{2n}, 
 \end{equation*}
using the bound \eqref{eq:bound_fkn} and the explicit definition of the coefficients $ a_n $ given in \eqref{eq:sequence_an} we deduce that, fixed $ t>0 $,  the power law \eqref{eq:omega_power_law} converges in any compact set of $ \bR_+ $. Moreover for any $ k\geqslant 1 $ the function $ \omega_k $ is a distributional smooth solution of \eqref{eq:vorticity_polar_CP} which converges to zero as $ t\searrow 0 $ in the sense of distributions, concluding. 
  \end{proof}

  \textit{Proof of Theorem \ref{thm:existence_infinitely_many_sol_R2} :} Applying the radial Biot-Savart law \eqref{eq:BS_radial} to any one of the vorticities constructed in Proposition \ref{prop:construction_vorticity} we deduce the explicit power law defining $ u_k, \ k\geqslant 1 $
  \begin{equation}
  \label{eq:def_uk}
  \begin{aligned}
  u_k \pare{x, t} & =  x^\perp \sum_{n=0}^\infty \frac{a_n}{2\pare{n+1}} \ f^{\pare{n}}_k\pare{t} r^{2n}, \\
  & =  x^\perp \sum_{n=0}^\infty \frac{a_n}{2\pare{n+1}} \ f^{\pare{n}}_k\pare{t} \pare{x_1^2 + x_2^2}^{n}, 
  \end{aligned}
  \end{equation}
  and $ f_k $ is defined as in \eqref{eq:def_fk}. Let us hence consider the Poisson equation in the unknown $ p_k $:
 \begin{equation}\label{eq:Poisson_eq_pressure}
 \Delta p_k = \div\pare{u_k\cdot\nabla u_k}.
 \end{equation}

We underline the fact that being $ u_k $ smooth and locally $ L^\infty $ then the distribution $  \div\pare{u_k\cdot\nabla u_k} $ is well defined, smooth and locally $ L^\infty $. 

\begin{claim} \label{claim:pressure_radial}
Let $ u_k $ be defined as in \eqref{eq:def_uk}, then the pressure defined by the Poisson equation \eqref{eq:Poisson_eq_pressure} is a radial distribution. 
\end{claim}

Let us hence prove the Claim \ref{claim:pressure_radial}, we drop the index $ k $ for the sake of clarity. If $ u $ is defined as in \eqref{eq:def_uk} it is hence clear that
\begin{equation*}
u\pare{x, t} = \pare{\begin{array}{c}
-x_2 \\x_1
\end{array}} U\pare{r, t}, 
\end{equation*}
where $ U $ is a radial distribution. Moreover standard computations imply that, if $ \div\ u=0 $; 
\begin{equation*}
\div\pare{u\cdot\nabla u} = \pare{\partial_1 u_1}^2 + \pare{\partial_2u_2}^2 + 2 \partial_1 u_2 \  \partial_2 u_1. 
\end{equation*}
Whence since 
\begin{align*}
x_1 & = \cos \varphi \ r, & x_2 &= \sin\varphi \ r, \\
\partial_1 & = \cos \varphi \ \partial_r , &
\partial_2 & = \sin \varphi \ \partial_r ,
\end{align*}
we can argue that
\begin{equation*}
\pare{\partial_1 u_1}^2 + \pare{\partial_2u_2}^2 =
2 r^2 \cos^2 \varphi \sin^2 \varphi \pare{\partial_r U}^2, 
\end{equation*}
while since
\begin{equation*}
 \partial_1 u_2 \  \partial_2 u_1 = - U^2 -  r^2 \cos^2 \varphi \  \sin^2 \varphi \pare{\partial_r U}^2 - r \ U\partial_r U  \cos^2\varphi - r \ U\partial_r U  \sin^2\varphi, 
\end{equation*}
we deduce that
\begin{align*}
\div\pare{u\cdot\nabla u} & = \pare{\partial_1 u_1}^2 + \pare{\partial_2u_2}^2 + 2 \partial_1 u_2 \  \partial_2 u_1 , \\
& = -2 U^2 - 2r  \ U\partial_r U , 
\end{align*}
is a radial distribution. Therefore since the Laplacian is rotationally invariant the Poisson equation \eqref{eq:Poisson_eq_pressure} can hence be rephrased as the ODE
\begin{equation*}
 p_{rr} + \frac{1}{r}\  p_r  = -2 U^2 - 2r  \ U\partial_r U, 
\end{equation*}
which we can solve uniquely, up to a constant, by identification of coefficients. \\

 We hence identify infinitely many $ \pare{u_k, p_k}_k $  distributional solution of \eqref{NS0} which are not tempered  distributions, concluding. \hfill$ \Box $

 \section{Playing with Burgers equation}\label{sec:Burgers} In the example constructed for the bidimensional  \NS\ we exploited the fact that $ u\cdot \nabla\omega =0 $ for radial vector fields, whence we constructed infinitely many smooth solutions for the linear diffusion equation \eqref{eq:vorticity_polar_CP} using the technique of Section \ref{sec:Tyc}. In the present section instead we  linearize a nonlinear parabolic equation (namely Burgers equation) and we will produce infinitely many weak solutions on the renormalized unknown.  \\
 
 In the present section we consider the one dimensional Burgers initial value problem
 
 \begin{equation}\label{eq:Burgers1D}
 \left\lbrace
 \begin{aligned}
 & u_t +\frac{1}{2} \pare{u^2}_x = u_{xx}, & \pare{x, t}&\in \bR\times\pare{0, \infty} \\
 & \left. u\right|_{t=0}= 0, & x&\in \bR.
 \end{aligned}
 \right.
 \end{equation}
 
The result we prove in such section is the following one
\begin{prop}
There exist infinitely many $ u \in \cC^1\pare{\pare{0, \infty} ; \cD'\pare{\bR^2}} \cap L^\infty_{\loc}\pare{\bR^2\times \bR_+} $ smooth distributional solutions to the Burger equation \eqref{eq:Burgers1D} such that $ u\pare{\cdot , t}\to 0 $ in $ \cD'\pare{\bR} $ as $ t\to 0^+ $. 
\end{prop} 
 
 \begin{proof}
  It is well known that the Cole-Hopf transformation
 \begin{equation}\label{eq:HC_transform}
 u = -2 \ \frac{\phi_x}{\phi}, 
 \end{equation}
 linearize the Burgers equation, i.e $ \phi $ solves
 \begin{equation} \label{eq:heat_phi}
 \left\lbrace
 \begin{aligned}
 & \phi_t = \phi_{xx}, \\
 & \left.  \phi\right|_{t=0} =\phi_0 . 
 \end{aligned}
 \right.
 \end{equation}
 The value of $ \phi_0 $ in \eqref{eq:heat_phi} is indeed determined by the initial value of \eqref{eq:Burgers1D} via the transformation \eqref{eq:HC_transform}. In the present work we suppose $ \phi_0\pare{x}=1 $, we will see that such condition  in \eqref{eq:heat_phi} suffice to obtain solutions to \eqref{eq:Burgers1D} via the transform \eqref{eq:HC_transform}. Let us hence define $ \tilde{\phi} $ as 
\begin{equation*}
\phi\pare{x, t} = 1 + \tilde{\phi}\pare{x, t}, 
\end{equation*}
if $ \phi $ solves \eqref{eq:heat_phi} with initial datum $  \phi_0\pare{x}=1  $ then indeed $ \tilde{\phi} $ has to solve 
\begin{align*}
\tilde{\phi}_t = \tilde{\phi}_{xx},  && \left. \tilde{\phi}\right|_{t=0} =0. 
\end{align*}

\noindent As explained in Section \ref{sec:Tyc} we can provide infinitely many solutions $ \pare{\tilde{\phi}_k}_{k\geqslant 1} $ to the above system of the form
\begin{align*}
\tilde{\phi}_{k}\pare{x, t}=\sum_{n =0}^\infty \frac{f^{\pare{n}}_{k}\pare{t}}{\pare{2 n}!} \ x^{2n}, && 
f_k \pare{t}= e^{-t^{-2k}}. 
\end{align*}
Moreover 
\begin{equation*}
\tilde{\phi}_k, \pare{\tilde{\phi}_k}_x \xrightarrow{t \searrow 0}0, \text{ in } \cD' \pare{\bR}, 
\end{equation*}
whence setting $ \phi_k = 1+\tilde{\phi}_k $ and $ u_k = -2\phi_k^{-1} \pare{\phi_k}_x $ we deduce that for each $ k\geqslant 1 $ $ u_k $ solves \eqref{eq:Burgers1D} in the sense of distributions and $ u_k\pare{\cdot , t}\xrightarrow{t\searrow 0} 0 $ in $ \cD'\pare{\bR} $, concluding. 
 \end{proof}

\footnotesize{
\providecommand{\bysame}{\leavevmode\hbox to3em{\hrulefill}\thinspace}
\providecommand{\MR}{\relax\ifhmode\unskip\space\fi MR }
\providecommand{\MRhref}[2]{%
  \href{http://www.ams.org/mathscinet-getitem?mr=#1}{#2}
}
\providecommand{\href}[2]{#2}

}


\begin{thebibliography}{10}

\bibitem{BCD11}
Hajer Bahouri, Jean-Yves Chemin, and Rapha\"el Danchin, \emph{Fourier analysis
  and nonlinear partial differential equations}, Grundlehren der Mathematischen
  Wissenschaften [Fundamental Principles of Mathematical Sciences], vol. 343,
  Springer, Heidelberg, 2011. \MR{2768550}

\bibitem{BdLSV17}
Tristan Buckmaster, Camillo De~Lellis, Philip Isett, and L\'aszl\'o
  Sz\'ekelyhidi, Jr., \emph{{O}nsager's conjecture for admissible weak
  solutions}, \url{https://arxiv.org/abs/1701.08678}.

\bibitem{bdLIS15}
\bysame, \emph{Anomalous dissipation for {$1/5$}-{H}\"older {E}uler flows},
  Ann. of Math. (2) \textbf{182} (2015), no.~1, 127--172. \MR{3374958}

\bibitem{BV17}
Tristan Buckmaster and Vlad Vicol, \emph{{N}onuniqueness of weak solutions to
  the {N}avier-{S}tokes equation}, \url{https://arxiv.org/abs/1709.10033}.

\bibitem{dLS09}
Camillo De~Lellis and L\'aszl\'o Sz\'ekelyhidi, Jr., \emph{The {E}uler
  equations as a differential inclusion}, Ann. of Math. (2) \textbf{170}
  (2009), no.~3, 1417--1436.

\bibitem{dLS10}
\bysame, \emph{On admissibility criteria for weak solutions of the {E}uler
  equations}, Arch. Ration. Mech. Anal. \textbf{195} (2010), no.~1, 225--260.

\bibitem{dLS13}
\bysame, \emph{Dissipative continuous {E}uler flows}, Invent. Math.
  \textbf{193} (2013), no.~2, 377--407. \MR{3090182}

\bibitem{dLS14}
\bysame, \emph{Dissipative {E}uler flows and {O}nsager's conjecture}, J. Eur.
  Math. Soc. (JEMS) \textbf{16} (2014), no.~7, 1467--1505. \MR{3254331}

\bibitem{I16}
Philip Isett, \emph{A proof of onsager's conjecture},
  \url{https://arxiv.org/abs/1608.08301}.

\bibitem{ISS03}
L.~Iskauriaza, G.~A. Ser\"egin, and V.~\v{S}verak,
  \emph{{$L_{3,\infty}$}-solutions of {N}avier-{S}tokes equations and backward
  uniqueness}, Uspekhi Mat. Nauk \textbf{58} (2003), no.~2(350), 3--44.

\bibitem{KL57}
A.~A. Kiselev and O.~A. Lady\v{z}enskaya, \emph{On the existence and uniqueness
  of the solution of the nonstationary problem for a viscous, incompressible
  fluid}, Izv. Akad. Nauk SSSR. Ser. Mat. \textbf{21} (1957), 655--680.

\bibitem{LR99}
Pierre~Gilles Lemari\'e-Rieusset, \emph{Solutions faibles d'\'energie infinie
  pour les \'equations de {N}avier-{S}tokes dans {$\bold R^3$}}, C. R. Acad.
  Sci. Paris S\'er. I Math. \textbf{328} (1999), no.~12, 1133--1138.
  \MR{1701373}

\bibitem{MB02}
Andrew~J. Majda and Andrea~L. Bertozzi, \emph{Vorticity and incompressible
  flow}, Cambridge Texts in Applied Mathematics, vol.~27, Cambridge University
  Press, Cambridge, 2002. \MR{1867882}

\bibitem{P59}
Giovanni Prodi, \emph{Un teorema di unicit\`a per le equazioni di
  {N}avier-{S}tokes}, Ann. Mat. Pura Appl. (4) \textbf{48} (1959), 173--182.

\bibitem{S62}
James Serrin, \emph{The initial value problem for the {N}avier-{S}tokes
  equations}, Nonlinear {P}roblems ({P}roc. {S}ympos., {M}adison, {W}is.,
  1962), Univ. of Wisconsin Press, Madison, Wis., 1963, pp.~69--98.

\bibitem{T35}
A.~Tychonoff, \emph{Th\'eor\`emes d'unicit\'e pour l'\'equation de la chaleur},
  Jour Mat. Sb. (1935).

\end{thebibliography}
\end{document}